\documentclass[12pt]{article}
\usepackage[utf8]{inputenc}
\usepackage{amsmath,amssymb,amsthm}

\newtheorem*{statement*}{Statement}
\newtheorem{statement}{Statement}[section]
\newtheorem{theorem}{Theorem}
\newtheorem*{theoremA}{Theorem A}
\theoremstyle{remark}

\theoremstyle{definition}
\newtheorem*{question}{Question}

\newcommand\R{\mathbb R}
\newcommand\eps{\varepsilon}
\renewcommand\le{\leqslant}
\renewcommand\ge{\geqslant}

\DeclareMathOperator\Rig{Rig}
\DeclareMathOperator\rank{rank}
\DeclareMathOperator\sign{sign}

\date{}
\title{Matrix and tensor rigidity and $L_p$-approximation}
\author{Yuri Malykhin\thanks{Steklov Mathematical Institute. Email:
malykhin-yuri@yandex.ru}}

\begin{document}
\maketitle
\abstract{In this paper we apply methods originated in Complexity
theory to some problems of Approximation.

We notice that the construction of Alman and Williams that disproves the rigidity of
Walsh-Hadamard matrices,
provides good $\ell_p$-approximation for $p<2$. It follows that the 
first $n$ functions of Walsh system can be approximated with an error
$n^{-\delta}$ by a linear space of dimension $n^{1-\delta}$:
$$
d_{n^{1-\delta}}(\{w_1,\ldots,w_n\}, L_p[0,1]) \le n^{-\delta},\quad
p\in[1,2),\;\delta=\delta(p)>0.
$$
We do not know if this is possible for the trigonometric system.

We show that the algebraic method  of Alon--Frankl--R\"odl for
bounding the number of low-signum-rank matrices, works
for tensors: almost all signum-tensors have large signum-rank and can't be
$\ell_1$-approximated by low-rank tensors. This implies lower bounds
for $\Theta_m$~--- the error of $m$-term approximation of multivariate functions
by sums of tensor products $u^1(x_1)\cdots u^d(x_d)$. In particular, for the set
of trigonometric polynomials with spectrum in $\prod_{j=1}^d[-n_j,n_j]$ and of norm
$\|t\|_\infty\le 1$ we have
$$
\Theta_m(\mathcal T(n_1,\ldots,n_d)_\infty,L_1[-\pi,\pi]^d) \ge c_1(d)>0,\quad m\le
c_2(d)\frac{\prod n_j}{\max\{n_j\}}.
$$
Sharp bounds follow for classes of dominated mixed smoothness:
$$
    \Theta_m(W^{(r,r,\ldots,r)}_p,L_q[0,1]^d)\asymp m^{-\frac{rd}{d-1}},\quad\mbox 2\le p\le\infty,\;
1\le q\le 2.
$$
}


\section{Introduction}

In this paper we consider several problems of low-rank approximation that are
related both to Approximation theory and Complexity Theory. 
We obtain some new results in Approximation using methods originated in
Complexity.

In the Introduction we discuss related notions and state our main results.
The proofs and all details are given in the subsequent sections.


\paragraph{Matrix rigidity and widths.}

In Complexity Theory it is common to relate the compexity of computational tasks to
some measures of complexity of simpler objects, say, matrices or graphs.
One important measure of matrix complexity is its \textit{rigidity},
$\Rig(A,m)$~--- the minimum Hamming distance from $A$ to
matrices of rank $\le m$. That is, $\Rig(A,m)$ is the number of entries
that must be modified in $A$ to drop it's rank to $m$.
Matrix rigidity was introduced by Valiant~\cite{V77} as a way to
obtain lower bounds for linear circuits.
He proved that if a linear transformation $x\mapsto A_n x$ defined by
a $n\times n$ matrix can be computed by a log-depth circuit of size $O(n)$,
then $\Rig(A_n,O(n/\log\log n))\le n^{1+\eps}$ for every fixed $\eps>0$.
It follows that matrices with high rigidity, e.g.,
$\Rig(A_n,\eps n)\ge n^{1+\eps}$ for some fixed $\eps>0$, provide
a superlinear lower bound for the size of corresponding linear circuits.
Though, still no constructive family of Valiant-rigid matrices is known.
See the survey~\cite{L08} for more details.

In Approximation Theory one is usually interested in approximation within a distance
that comes from some norm, e.g., $\ell_p$ norm:
$\|x\|_{\ell_p^N}=(|x_1|^p+\ldots+|x_N|^p)^{1/p}$, $1\le p<\infty$,
$\|x\|_\infty=\max|x_i|$.

The case of element-wise $\ell_2$-norm (i.e., Frobenius norm) 
is well studied. The value of the best $m$-rank approximation is given by the
the Eckart--Young theorem in terms of singular values:
\begin{equation}
    \label{EY}
\min_{\rank B\le m}\sum_{i,j}|A_{i,j}-B_{i,j}|^2 = \sum_{k>m}\sigma_k(A)^2.
\end{equation}
It follows that any orthogonal matrix is ``$\ell_2$-rigid'', i.e. it can't be
well approximated in $\ell_2$ by low-rank matrices.  Indeed, such matrix has Frobenius
norm $n^{1/2}$, singular values $\sigma_k(A)\equiv 1$ and the distance to
any matrix of rank $\le m$ is at least $(n-m)^{1/2}$.

An interesting case is the maximum norm $\|A\|_\infty := \max_{i,j}|A_{i,j}|$.
There is a corresponding notion of the $\eps$-rank:
$$
\rank_\eps(A):=\min\{\rank B\colon \|A-B\|_\infty\le\eps\}.
$$
It appeared as a lower bound for the complexity of Quantum
Communication, see~\cite{BW01},~\cite{LS09} and \cite{ALSV13}.
On the other side, the error of the best $m$-rank approximation in
$\ell_\infty$-norm is a special case of a well-known concept~--- Kolmogorov
widths, defined by A.N.~Kolmogorov in 1936. Let us give the definition.

Let $X$ be a normed space and $W\subset X$. The Kolmogorov $m$-width of $W$ in
$X$ is the quantity
$$
d_m(W,X) := \inf_{\dim Q_m\le m} \sup_{x\in W}\rho(x, Q_m)_X,
$$
where $\inf$ is taken over linear subspaces $Q_m\subset X$ of $\dim\le m$ and $\rho$ is the distance to a set:
$\rho(x,Q)_X:=\inf_{y\in Q} \|x-y\|_X$.

The books~\cite[Ch.13, 14]{LGM96} and \cite{P85} are standard references on the
subject. See also a recent survey~\cite{DTU18}, \S4.3.

So, we have a simple equivalence (emphasized in~\cite{KMR18}):
$$
\rank_\eps(A) \le m \quad \Longleftrightarrow \quad
d_m(\{A_i\}_{i=1}^{n_1},\ell_\infty^{n_2}) \le
\eps
$$
for $A\in\R^{n_1\times n_2}$ and $A_i$ being the rows of $A$.

Hadamard matrices (square matrices with $\pm1$ entries and orthogonal rows) 
are ``$\ell_\infty$-rigid'': it follows from the $\ell_2$ case 
that $\|A-B\|_\infty \ge n^{-1}(n-m)^{1/2} \ge (1-\eps)^{1/2}$ if $A$ is Hadamard and $\rank B\le\eps n$.
However, the problem of estimating $\eps$-rank is difficult
in many other cases.  An interesting example is the family of upper-triangular
0/1 matrices $\Delta_n$; it is known only that (see~\cite{ALSV13} and the
discussion in~\cite{KMR18})
$$
c\log^2n \le \rank_{1/3}(\Delta_n) \le C\log^3n.
$$

The $\ell_1$ case is less studied.
It is an interesting question about ``$\ell_1$-rigid'' matrices: can one give an example of constructive
family of $n\times n$ matrices with, say, $\pm1$ entries, such that
\begin{equation}
    \label{l1rig}
    \min_{\rank B\le \eps n}\sum_{i,j=1}^n|A_{i,j}-B_{i,j}|\ge \eps n^2
\end{equation}
for some small $\eps>0$? We will see that almost all
signum matrices are $\ell_1$-rigid, but the proof is not
constructive.\footnote{We mention in this regard Open Problem 5.6
from~\cite{T03}, that was a starting point for our research. \textit{Let $R_N$
be the Rudin--Shapiro polynomials.
Prove that $\Theta_m(R_N(x-y),L_1[-\pi,\pi]^2)\gg N^{1/2}$.}}

It would be interesting to relate the (original) rigidity to
$\ell_p$-approximation and widths.
We give one particular example of such a connection.

It was conjectured that Walsh--Hadamard matrices
$H^k\in\{-1,1\}^{n\times n }$, $n=2^k$, are Valiant-rigid. To define $H^k$ it is
convenient to
index its rows and columns by boolean vectors
$x,y\in\{0,1\}^k$:
$$
H^k_{x,y} = (-1)^{\sum_{i=1}^n x_iy_i}.
$$
It was a surprising result due to J.~Alman and R.~Williams~\cite{AW17}
that $H^k$ are not rigid. They proved that one can change
$n^{c\eps\log(1/\eps)}$ entries in each row of $H^k$ and drop its rank below
$n^{1-c\eps^2}$. 

We prove that the construction of~\cite{AW17} gives a matrix
with not large entries (see Statement~\ref{stm_walsh_hadamard}). 
It follows that this matrix provides nontrivial $\ell_p$
approximation for rows of $H^k$ for $p<2$.
Rows of $H^k$ correspond to the values of Walsh functions (in the Paley numeration)~--- the well-known
orthonormal system on $[0,1]$, see~\cite[\S1.1]{GES}.
So, we can formulate the approximation property in terms of Kolmogorov widths.
\begin{theorem}
    \label{th_walsh}
    Let $w_1,w_2,\ldots$ be the Walsh functions. For any $p\in[1,2)$
    there exists $\delta=\delta(p)>0$ such that for sufficiently large $n$ the
    inequality holds
    $$
    d_{n^{1-\delta}}(\{w_1,\ldots,w_n\},L_p[0,1])\le n^{-\delta}.
    $$
\end{theorem}

Note that good approximation in $L_2$ is impossible.  In fact, for any
orthonormal system $\varphi_1,\ldots,\varphi_n$ in $L_2[0,1]$ one has\footnote{
The last equality in this chain follows from~\eqref{EY}. The second equality
in terms of widths appeared in a paper of S.B. Stechkin (1954) and in terms of
matrices~--- in a paper of A.I.~Maltsev (1947). See~\cite{Tikh87} for references
and historical details.}
\begin{multline*}
    d_m(\{\varphi_1,\ldots,\varphi_n\},L_2[0,1]) =
    \inf_{\dim Q_m\le m}\max_{1\le k\le n}\rho(\varphi_k,Q_m)_{L_2} = \\
    = \inf_{\dim Q_m\le m}\left(\frac1n\sum_{k=1}^n\rho(\varphi_k,Q_m)_{L_2}^2\right)^{1/2} =
     (1-m/n)^{1/2}.
\end{multline*}

It would be interesting to consider the $L_1$-``rigidity'' for other orthonormal
systems.

\begin{question}
    Is there a good $L_1$-approximation of the trigonometric system:
$$
d_{o(n)}(\{\exp(ikx)\}_{k=1}^n,L_1[-\pi,\pi]) = o(1)\;?
$$
\end{question}
This seems probable. Note that it was proven recently by Dvir and
Liu~\cite{DL19} that the discrete Fourier matrices are not rigid:
$$
\Rig(F^n,n\exp(-c_\eps\log^c n)) \le n^{1+\eps},
$$
for any $\eps>0$ and sufficiently large $n$. (In fact, they prove that the
number of changes is at most $n^\eps$ in each row and column.)
The proof is much more complicated than for Walsh--Hadamard and it is hard to
figure out upper bounds for elements of approximation matrices.

Some orthonormal systems are $L_1$-rigid. For example, good approximation of
Rademacher system is impossible even in average, see
Statement~\ref{stm_rademacher}.

\paragraph{Signum rank and multivariate functions approximation.}

There is another measure of matrix complexity called signum-rank. Let $S$ be a
signum matrix, i.e. matrix with $\pm1$ entries. Its signum rank, $\rank_\pm(S)$,
is defined as the minimal rank of matrix $T$ with $\sign T_{i,j}\equiv S_{i,j}$.
We suppose that $\sign(0)=0$, so $T$ must have nonzero entries. This notion
was introduced by R.~Paturi and J.~Simon~\cite{PS86} in the context of unbounded
error probabilistic communication complexity.

It is highly nontrivial that there are matrices with high signum rank; this was
proven by Alon, Frank and R\"odl in~\cite{AFR85}. Let us describe their method.

Let $p_1,\ldots,p_l$ be polynomials in $\R^k$ of degree at most $s$.
Denote by $Z(p_1,\ldots,p_l)$ the cardinality of the set of vectors
$(\sign p_1(x),\ldots,\sign p_l(x))$ with non-zero coordinates.
Let $Z_{\max}(k,l,s)$ be the maximal possible value of $Z(p_1,\ldots,p_l)$.
Warren~\cite{W68} proved that this number is less or equal than
$(4esl/k)^k$ for $l\ge k$. (He applied this inequality to get lower bounds for
the approximation error of functional classes by certain ``polynomial sets''.)

In~\cite{AFR85} the following observation was made. If $S$ is a $n_1\times n_2$ signum matrix,
$\rank_\pm S\le m$, then $S_{i,j}=\sign \sum_{k=1}^m u_{k,i}v_{k,j}$ for some
$u_k\in\R^{n_1}$, $v_k\in\R^{n_2}$. Hence the number of such matrices does not
exceed $Z(\{p_{i,j}\})$, where $p_{i,j}(x,y)=\sum_{k=1}^m x_{k,i}y_{k,j}$ are
degree-2 polynomials of $m(n_1+n_2)$ variables.

Matrix-related methods often fail for tensors, e.g., there is no good analog of
Singular Vector Decomposition. We observed that the method of~\cite{AFR85}
applies to tensors as well. The definition of tensor signum-rank is analogous to
that of matrices. Let $S_m(n_1,\ldots,n_d)$ be the cardinality of signum tensors
in $\{-1,1\}^{n_1\times \ldots\times n_d}$ having signum-rank at most $m$. Then the
above reasoning applied to tensors gives us the following.

\begin{statement*}
    $S_m(n_1,\ldots,n_d) \le Z_{\max}(m(n_1+\ldots+n_d), n_1\cdots n_d, d)$.
\end{statement*}

Together with Warren's bound this proves that almost all signum--tensors have
high signum rank. Moreover, allmost all of them can't be well--approximated in the
$\ell_1$ norm by low-rank tensors.

This approach leads to new lower bounds in $L_1$ for the continuous analog of
low-rank approximation, namely, $m$-term approximation of
multivariate functions by tensor products. Let $X=\prod_{i=1}^d [a_i,b_i]$.
Define for $f\in L_p(X)$
\begin{equation}
    \Theta_m(f,L_p) := \inf_{u^{i,s}\colon [a_i,b_i]\to\mathbb R} \|f - \sum_{s=1}^m u^{1,s}(x_1)\cdots
u^{d,s}(x_d)\|_{L_p(X)}
\end{equation}
and $\Theta_m(\mathcal F,L_p) := \sup_{f\in\mathcal F}\Theta_m(f,L_p)$.

Using discretization and above arguments on tensors, one can lower bound
$\Theta_m$ for various function classes.

To work with classes of fractional
smoothness it is convenient to consider classes of trigonometric polynomials.
Let $\mathbf{n}=(n_1,\ldots,n_d)\in\mathbb N^d$ and denote by $\mathcal T(\mathbf{n})_\infty$
the set of real trigonometric polynomials satisfying $\|t\|_\infty\le 1$ and
having the spectrum in $\prod_{j=1}^d[-n_j,n_j]$.
\begin{theorem}
    \label{thm_poly}
    For any $d\in\mathbb N$, $\mathbf{n}=(n_1,\ldots,n_d)$,
    we have
    \begin{equation}
    \Theta_m(\mathcal T(\mathbf{n})_\infty, L_1[-\pi,\pi]^d) \ge
        c_2(d)>0,\quad\mbox{if $m\le c_1(d)\frac{\prod n_j}{\max\{n_j\}}$.}
    \end{equation}
\end{theorem}

There is a well-known result of V.~Temlyakov~\cite{T89} providing the upper bound
for the classes of functions with bounded mixed derivative
$\mathbf{r}=(r,r,\ldots,r)$:
$$
\Theta_m(W^{\mathbf{r}}_2,L_2[0,1]^d) \ll m^{-\frac{rd}{d-1}}.
$$
Using Theorem~\ref{thm_poly} we can get a sharp lower bound
$$
\Theta_m(W^{\mathbf{r}}_\infty,L_1[0,1]^d) \gg m^{-\frac{rd}{d-1}}.
$$
Putting these bounds together, we obtain the following result.
\begin{theorem}
    \label{th_mixed}
    For $d\in\mathbb N$, $r>0$, $\mathbf{r}:=(r,\ldots,r)$, we have
    $$
    \Theta_m(W^{\mathbf{r}}_p,L_q[0,1]^d)\asymp m^{-\frac{rd}{d-1}},\quad\mbox 2\le p\le\infty,\;
    1\le q\le 2.
    $$
\end{theorem}
This strenghtens the results of Temlyakov and Bazarkhanov,~\cite{BT15}.
Note that the case $d=2$ is well studied and the corresponding result is known, see Theorem 1.1 of~\cite{T92};
the novelty of our results is the sharp bounds for $d>2$.

\section{Approximation of Walsh and Rademacher functions}
\label{section_walsh}

We will prove the following bound.
\begin{statement}
    \label{stm_walsh_hadamard}
    Let $\eps\in(0,\frac12)$. For any sufficiently large $k$ there exists a matrix
    $B\in\R^{n\times n}$, $n=2^k$, such that the conditions hold:
    \begin{itemize}
        \item[(i)] in each row $B$ differs from $H^k$ in at most
            $n^{c_1\eps\log(1/\eps)}$ entries,
        \item[(ii)] $\rank B \le n^{1-c_2\eps^2}$,
        \item[(iii)] $\|B\|_\infty \le n^{\frac12+c_3\eps\log(1/\eps)}$.
    \end{itemize}
\end{statement}

Using this Statement, we can give an upper bound on the $\ell_p$-approximation
of rows of $H^k$:

\begin{multline*}
    (n^{-1}\sum_{y}|B_{x,y}-H^k_{x,y}|^p)^{\frac1p} \le
    n^{-\frac1p}(1 + \|B\|_\infty)\cdot |\{y\colon B_{x,y}\ne
    H^k_{x,y}\}|^{\frac1p} \le \\
    \le n^{-\frac1p+\frac12+c\eps\log(1/\eps)} \le n^{-\delta},\quad\mbox{if
    $p<2$ and $\eps<\eps(p)$}.
\end{multline*}

Theorem~\ref{th_walsh} follows.

Let us prove Statement~\ref{stm_walsh_hadamard}.
We will use the standard bound for binomial coefficients in terms of the entropy function
$h_2(p)=-p\log_2p-(1-p)\log_2(1-p)$:
\begin{equation}
    \label{binom_bound}
    \sum_{j=0}^s\binom{k}{j} \le 2^{k h_2(s/k)},\quad s\le k/2.
\end{equation}

\begin{proof}
    The construction almost repeats~\cite{AW17}, where the properties (i)
    and (ii) are established. We write down the essential steps to prove (iii).
    Given $k$ and $\eps$, we define
    $$
    l := \lceil 2k\eps \rceil - 1,\quad d := \lceil (1/2-\eps)k \rceil,
    $$
    and the ``core set''
    $$
    \mathcal C_{k,\eps}:=\{x\in\{0,1\}^k\colon
    \|x\|_1\in[(1/2-\eps)k,(1/2+\eps)k]\}.
    $$

    Step 1. We need a polynomial $q(z_1,\ldots,z_k)$, $\deg q\le d$, such that
    $$
    q(z)=(-1)^{\|z\|_1},\quad\mbox{if $z\in\{0,1\}^n$ and $l+1\le\|z\|_1\le
    l+d+1$.}
    $$
    In~\cite{AW17} the polynomial was taken in $\mathbb Z[z_1,\ldots,z_k]$ to
    prove non-rigidity over arbitrary field. As we are interested in the
    case of field $\R$, we simply take $q(z)=(-1)^{l+1}Q_d(z_1+\ldots+z_k-l-1)$ with
    $$
    Q_d(j)=(-1)^j,\;j=0,1,\ldots,d,\quad Q_d\in\mathbb R[t],\,\deg Q_d\le d.
    $$

    Step 2. Put $p(x,y):=q(x_1 y_1,\ldots,x_k y_k)$. We replace 
    $x_i^my_i^m\mapsto x_iy_i$ in $p$ and get another polynomial
    $\widetilde{p}$; however, $\widetilde{p}$ coincides with $p$ on
    $\{0,1\}^k\times \{0,1\}^k$. The number $M$ of monomials in
    $\widetilde{p}$ is bounded by
        \begin{equation}
            \label{monom_bound}
            M(\widetilde{p})\le \sum_{s=0}^{\deg q}\binom{k}{s} \le
            2^{kh_2(d/k)}\le 2^{k(1-c\eps^2)}.
        \end{equation}

    Step 3. Put $\widetilde{B}_{x,y}=\widetilde{p}(x,y)$. The crucial
    observation is that
    \begin{equation}
        \label{rank_monom}
    \rank\widetilde{B}\le M(\widetilde{p}).
    \end{equation}
        By construction, $\widetilde{B}_{x,y}=H^k_{x,y}$, if $\langle x,y\rangle \in [2\eps k,(1/2+\eps)k]$.

    Step 4. Fix $\widetilde{B}$: leave $B_{x,y}:=\widetilde{B}_{x,y}$, if
        $x,y\in \mathcal C_{k,\eps}$, and for other pairs $(x,y)$
        define $B_{x,y}:=H^k_{x,y}$. Changes occur in a small number of
        rows and columns, so the rank of $B$ is bounded from above:
        \begin{equation}
            \label{changes}
            \rank B \le \rank\widetilde{B} + 2|\{0,1\}^k\setminus \mathcal C_{k,\eps}|.
        \end{equation}

    The matrix $B$ is our approximation matrix. The property (ii) follows
    from~(\ref{monom_bound}), (\ref{rank_monom}) and (\ref{changes}). Note that
    $B$ can differ from $H^k$ only for $x,y\in \mathcal C_{k,\eps}$ with $\langle
    x,y\rangle\not\in[2k\eps,(1/2+\eps)k]$; obviosly, such pairs $(x,y)$ satisfy the following
    conditions:
    \begin{equation}\label{bad}
        \left\{
            \begin{aligned}
                \langle x,y\rangle < 2\eps k,\\
                x,y \in \mathcal C_{k,\eps}.
            \end{aligned}
            \right.
    \end{equation}

    Lemma 3.3~\cite{AW17} shows that for each $x$ there are less than
    $2^{ck\eps\log(1/\eps)}$ vectors $y$ such that~(\ref{bad}) holds true. The
    property (i) follows. 

    Now we will bound $\max|B_{i,j}|$. It's values either coincide with $H^k$,
    or come from the polynomial $p$. Hence,

\begin{multline}\label{max}
\|B\|_\infty \le 1 + \max_{(x,y)\colon(\ref{bad})} |p(x,y)| \le \\
    1 + \max_{z\in\{0,1\}^{n}\colon \sum z_i < 2\eps k} |q(z)| \le
    1 + \max\limits_{s=1,\ldots,l+1}|Q_d(-s)|.
\end{multline}

Newton interpolation allows us to write $Q_d$ in explicit form:
$$
Q_d(t) = \sum_{j=0}^d(-2)^j\binom{t}{j},\quad\mbox{where
}\binom{t}{j}=\frac{t(t-1)\cdots(t-j+1)}{j!}.
$$
We have
$
\left|\binom{-s}{j}\right| = \binom{s+j-1}{j} \le \binom{l+j}{j} = \binom{l+j}{l} \le \binom{l+d}{l}.
$
Hence, $|Q_d(-s)|\le 2^{d+1}\binom{l+d}{l}$.
Finally, using~(\ref{binom_bound}), we obtain from~\eqref{max}:
\begin{multline*}
    \|B\|_\infty \le 1+\max_{1\le s \le l+1}|Q(-s)| \le 1+2^{d+1}\binom{l+d}{l} \le \\
    \le 2^{k(\frac12-\eps) + (l+d)h_2(\frac{l}{l+d}) + O(1)} \le \\
    \le 2^{k(\frac12-\eps)+k(\frac12+\eps)h_2(4\eps) + O(1)} \le
    2^{k(\frac12 +c\eps\log(1/\eps))}.
$$
\end{multline*}
\end{proof}

Let us return to the approximation of functions. We seen that the Walsh system
admits good approximation in $L_p$, $p<2$. Now we will prove that Rademacher
system can't be well approximated in $L_1$ even in average. Recall that the
Rademacher system $\{r_k\}$ is the orthogonal system on $[0,1]$ defined in terms
of binary expansions: $r_k(\sum\limits_{j\ge1} 2^{-j}x_j)=(-1)^{x_k}$. 

\begin{statement}
    \label{stm_rademacher}
    Let $\eps\in(0,1)$. For any linear space $Q_m\subset
    L_1[0,1]$ of dimension $m\le (1-\eps)n$ we have
    $$
    \frac1n\sum_{k=1}^n \rho(r_k, Q_m)_{L_1[0,1]} \ge c\eps^2.
    $$
\end{statement}

\begin{proof}
    Consider the
    matrix $M$ with $n$ rows indexed by $k\in\{1,\ldots,n\}$ and
    $2^n$ columns indexed by binary vectors $x\in\{0,1\}^n$:
    $M_{k,x} := (-1)^{x_k}$. This corresponds to the matrix of values of
    $r_1,\ldots,r_n$ on the dyadic intervals $(\frac{j-1}{2^n},\frac{j}{2^n})$.
    Suppose that the functions $r_1,\ldots,r_n$ are
    approximated by a subspace $Q_m$; pick optimal approximating functions
    $g_1,\ldots,g_n$. We can average them on the dyadic intervals (this does not
    make the approximation worse) and obtain the following discretization:
    \begin{multline*}
    \inf_{\dim Q_m\le m} \sum_{k=1}^n \rho(r_k,Q_m)_{L_1[0,1]} =
        \inf_{\dim \mathrm{span}\,\{g_1,\ldots,g_n\}\le m} \sum_{k=1}^n \|r_k-g_k\|_1 = \\
    \inf_{\rank B\le m} 2^{-n}\sum_{k,x}|M_{k,x}-B_{k,x}| =
    \inf_{\dim R_m\le m} 2^{-n}\sum_{x\in\{-1,1\}^n}\rho(x,R_m)_{\ell_1^n}.
    \end{multline*}
    It is well-known\footnote{This is a particular case of the equality
    $d_n(B_p^N,\ell_q^N)=(N-n)^{1/q-1/p}$, $p\ge q$, proved independently by
    M.I.~Stesin (1975) and A.~Pietsch (1974), see~\cite[Ch.VI]{P85} for references and
    more details.} that one can't approximate all vertices of the cube
    $B_\infty^m := [-1,1]^m$:
    \begin{equation}
        \label{cube_width}
        d_m(B_\infty^n,\ell_1^n) = n-m,
    \end{equation}
    but we need an averaged lower bound. Fix $R_m\subset\R^n$ and let $G$ be the ``best''
    half of vertices, i.e. the subset of $\{-1,1\}^n$ of cardinality $2^{n-1}$
    such that $\rho(v,R_m)_{\ell_1^n}\ge\rho(u,R_m)_{\ell_1^n}$ for $v\not\in
    G$, $u\in G$. Let us apply the result from~\cite[Theorem 1]{ST89}:

    \begin{theoremA}
        Let $D\subset\{-1,1\}^n$ be such that $|D|\ge 2^{n-1}$ and let
        $\eps\in(0,1)$. Then there exists $\sigma\subset\{1,2,\ldots,n\}$,
        $|\sigma|\ge (1-\eps)n$, safisfying
        $$
        \mathrm{absconv}(P_\sigma D)\supset c\eps[-1,1]^\sigma,
        $$
        where $P_\sigma$ denotes the restriction map $(a_j)_{1\le j\le n}\mapsto
        (a_j)_{j\in\sigma}\in\R^\sigma$ and $c$ is a numerical constant.
    \end{theoremA}
    Here $\mathrm{absconv}K := \mathrm{conv}(K\cup(-K))$; we note that
    $d_m(K,X)=d_m(\mathrm{absconv}K,X)$.

    We apply this theorem for $D:=G$ and $\eps/2$ to find a cube of
    dimension $|\sigma|\ge (1-\eps/2)n$. Then, using~(\ref{cube_width}), we finish the
    proof:
    \begin{multline*}
        2^{1-n}\sum_{u\in\{-1,1\}^n}\rho(u,R_m)_{\ell_1^n}
        \ge \max_{u\in G}\rho(u,R_m)_{\ell_1^n}
        \ge \max_{u\in G}\rho(P_\sigma u,P_\sigma R_m)_{\ell_1^\sigma} \ge \\
        \ge d_m(P_\sigma G,\ell_1^\sigma)
        \ge d_m(c\eps B_\infty^\sigma,\ell_1^\sigma)
        \ge c\eps\cdot((1-\frac{\eps}2)n-m) \ge c_1 n\eps^2.
    \end{multline*}
    Note that one does not require Theorem A if $\eps>1/2$; standard
    VC-dimension bounds will suffice.
\end{proof}
We see that the matrix $(M_{k,x})$ is $\ell_1$-rigid, but it is not square, as
required in~(\ref{l1rig}).

\section{Tensor signum rank and $\ell_1$-approximation}
\label{secttion_tensor}

We use the standard notation $[n]:=\{1,2,\ldots,n\}$.
We identify tensors with multidimensional arrays $T\in \R^{n_1\times
n_2\times\ldots\times n_d}$ consisting of numbers $(T_{i_1,\ldots,i_d})$
indexed by tuples $I=(i_1,\ldots,i_d)$, $i_1\in [n_1],\ldots,i_d\in[n_d]$. Such
tensor has
$$
N=n_1\ldots n_d
$$
elements. A tensor $T\ne 0$ has rank $1$ if
$$
T_{i_1,\ldots,i_d} = u^1_{i_1}u^2_{i_2}\cdots u^d_{i_d}
$$
for some vectors $u^1\in\R^{n_1},\ldots,u^d\in\R^{n_d}$.
Tensor has rank at most $m$, if it can be represented as a sum of $m$ rank-one
tensors. It is clear that a tensor from $\R^{n_1\times\ldots\times n_d}$ has rank at most
$N/\max\{n_j\}$: indeed, let $n_1$ be the maximal dimension, then
$$
T_{i_1,\ldots,i_d} =
\sum_{(i_2',\ldots,i_d')}T_{i_1,i_2',\ldots,i_d'}\delta_{i_2'}^{i_2}\cdots\delta_{i_d'}^{i_d}
$$
and each summand has rank one.

Let $S$ be a signum tensor, i.e. it consists of $\pm1$ elements. 
The signum rank of $S$ is the minimal rank of a tensor $T$ such that
$\sign T \equiv S$. It is denoted as $\rank_\pm S$. We suppose that $\sign 0=0$, so $T$ must have nonzero
elements. This is a straighforward generalization of the matrix case.

We show (see Statement~\ref{stm_low}) that almost all signum-tensors have high
signum rank and cannot be approximated in $\ell_1$ by low-rank tensors. Let us
connect signum rank and $\ell_1$-approximation.
There is an obvious inequality (for $x,y\in\R^N$):
$$
\|x-y\|_{\ell_1^N}\ge |\{i\colon \sign x_i \ne \sign y_i\}|\cdot
\min_{1\le i\le N}|x_i|.
$$
Hence, the $\ell_1$-distance from a tensor $T$ to a low-rank tensor $\widetilde
T$ is at least the Hamming distance between $\sign T$ and $S:=\sign\widetilde T$.
By definition $S$ has low signum rank, i.e. $\rank_\pm S\le\rank \widetilde{T}$.
There is a technical subtlety: if $\widetilde{T}$ has some zero entries
then $S$ is not a signum tensor; we can vary $\widetilde{T}$ a bit using
rank-one tensors and get rid of zero elements. So,
\begin{equation}
    \label{simple_bound}
    \inf_{\rank\widetilde{T}\le m} \|T-\widetilde{T}\|_{\ell_1^N} \ge
    \inf_{\rank_\pm S\le m+1}|\{I\colon\sign T_I\ne S_I\}|\cdot\min_I|T_I|.
\end{equation}

Denote by $\mathcal S_m(n_1,\ldots,n_d)$ the set of signum-tensors
$S\in\{-1,1\}^{n_1\times\cdots\times n_d}$ with $\rank_\pm S\le m$.
The cardinality of this set is denoted by
$S_m(n_1,\ldots,n_d)$.

Let $p_1,\ldots,p_l$ be some polynomials in $\R^k$ of degree at most $s$.
They define the set
$$
\Omega := \R^k\setminus\bigcup_{i=1}^l\{x\colon p_i(x)=0\}.
$$
Each point $x\in\Omega$ yields the signum vector $(\sign
p_1(x),\ldots,\sign p_l(x))\in\{-1,1\}^l$. Denote by $Z(p_1,\ldots,p_l)$ the
cardinality of the set of such signum-vectors; it is obvious that $Z$ does not
exceed the number of connected components of $\Omega$.
Warren~\cite{W68} proved that this number is less or equal than
$(4esl/k)^k$ for $l\ge k$.

If we denote by $Z_{\max}(k,l,s)$ the maximal possible value of $Z(p_1,\ldots,p_l)$
then Warren's bound implies
\begin{equation}
    \label{warren}
    Z_{\max}(k,l,s)\le (4esl/k)^k,\quad\mbox{for $l\ge k$.}
\end{equation}

The following statement was proven in~\cite{AFR85} for matrices.
\begin{statement}
    \label{stm_alon}
    $S_m(n_1,\ldots,n_d) \le Z_{\max}(m(n_1+\ldots+n_d), n_1\cdots n_d, d)$.
\end{statement}

\begin{proof}
    If a signum tensor $S$ has signum rank at most $m$, then it is represented
    as $S=\sign T$ for some tensor $T$ with $\rank T\le m$. 
    We have $T=\sum_{s=1}^m u^{1,s}\otimes\cdots\otimes u^{d,s}$ 
    for some vectors $u^{i,j}$. Hence the following equalities hold:
    \begin{equation}
        \label{rank_eq}
    S_{i_1,\ldots,i_d} = \sign(\sum_{s=1}^m u^{1,s}_{i_1} u^{2,s}_{i_2}
    \cdots u^{d,s}_{i_d}).
    \end{equation}

    Let us take the following point of view: there are variables
    $x^{k,s}_i$, $k\in [d]$, $i\in [n_k]$, $s\in [m]$; totally
    $m(n_1+\ldots+n_d)$ of them. There are also polynomials in these variables:
    $$
    q_{i_1,\ldots,i_d}(x) = \sum_{s=1}^m x^{1,s}_{i_1}x^{2,s}_{i_2}\cdots
    x^{d,s}_{i_d}.
    $$
    Given $S$ there should exist a value $x=u$,
    such that the vector $(\sign q_{i_1,\ldots,i_d}(u))$ equals $S$.
    Therefore,
    $$
    S_m(n_1,\ldots,n_d) \le Z(\{q_{i_1,\ldots,i_d}\}).
    $$
    It remains to use the definition of $Z_{\max}$.
\end{proof}

\begin{statement}
    \label{stm_low}
    Let $d\in\mathbb N$. For any sufficiently small $\eps>0$ and any numbers
    $n_1,\ldots,n_d$
    there is a set of signum tensors $\mathcal
    S^*\subset\{-1,1\}^{n_1\times n_2\times\ldots\times n_d}$  of
    cardinality
    $$
    |\mathcal S^*|\ge 2^N - 2^{N c(d)\eps\log(1/\eps)},\quad N:=n_1n_2\cdots n_d,
    $$
    with the following property. If a tensor $T$ has $\sign T\in\mathcal S^*$, then
    \begin{equation}
        \label{eq_low}
    \inf_{\rank \widetilde{T}\le \eps N/\max\{n_j\}}
    \|T-\widetilde{T}\|_{\ell_1^N} \ge \eps
    N\min_{i_1,\ldots,i_d}|T_{i_1,\ldots,i_d}|.
    \end{equation}
\end{statement}
We remark that all signum tensors $S\in\mathcal S^*$ have $\rank_\pm S > \eps
N/\max\{n_j\}$.

\begin{proof}
    Let $t=\eps N/\max\{n_j\}$. 
    If $t<1$, then $\rank\widetilde{T}\le t$ implies that $\widetilde{T}=0$
    and~\eqref{eq_low} is obvious; so we consider the case $t\ge1$.
    
    Put $m=\lceil t\rceil$ and $k := \lfloor \eps N\rfloor$. W.l.o.g. we assume
    that $t\not\in\mathbb Z$, so the inequality $\rank\widetilde{T}\le t$
    implies that $\rank\widetilde{T}\le m-1$.
    Consider the set $\mathcal S_m^k$ of signum-tensors that differ from
    tensors of $\mathcal S_m(n_1,\ldots,n_d)$ in at most $k$ entries. The set $\mathcal
    S^*$ is just the complement of $\mathcal S_m^k$. The
    inequality~\eqref{eq_low} follows from~(\ref{simple_bound}).

    Let us bound the cardinality of $\mathcal S_m^k$, using Statement~\ref{stm_alon}:
$$
    S_m(n_1,\ldots,n_d) \le Z_{\max}(m\sum n_j, N, d).
    $$
    We have $m\sum n_j\le 2t\sum n_j\le 2\eps Nd$, so
    Warren's bound~(\ref{warren}) gives us
    $$
    S_m(n_1,\ldots,n_d) \le (\frac{4edN}{2\eps dN})^{2\eps dN} \le
    2^{Nc(d)\eps\log(1/\eps)}.
    $$
    Further,
    $$
        |\mathcal S_m^k| \le S_m \sum_{j\le k}\binom{N}{j} \le S_m
        2^{N h_2(k/N)} \le S_m 2^{N h_2(\eps)}.
    $$
    Finally,
    $\log |\mathcal S_m^k| \le N c(d) \eps\log(1/\eps)$, as required.
\end{proof}

\section{Low-rank approximation of smooth functions}

Let $(X,\mu)$ be a measurable space, $\mathcal F$ be a set of real functions in
$L_1(X,\mu)$, $\mathcal A$ be a finite family of measurable subsets of $X$ and $\gamma>0$. We say that 
$\mathcal A$ is $\gamma$-shattered by 
$\mathcal F$ (with respect to measure $\mu$) if for any choice of signs
$\sigma\colon\mathcal A\to\{-1,1\}$ there is a function $f\in\mathcal F$ such
that
$$
\sigma(A)\cdot \int_A f(x)\,d\mu \ge \gamma,\quad \forall A \in \mathcal A.
$$

This definition is an integral version of fat-shattering dimension, a
generalization of Vapnik--Chervonenkis dimension for real-valued functions,
see the book~\cite{VHA15} for details.

Note that if sets in $\mathcal A=\{A_1,\ldots,A_N\}$ are pairwise disjoint,
then the discretization operator
$$
\mathcal I_{\mathcal{A}}\colon L_1(X,\mu)\to\R^N,\quad f\mapsto
(\int_{A_i}f\,d\mu)_{i=1}^N,
$$
satisfies
\begin{equation}
    \label{I_bound}
    \|f\|_{L_1(X,\mu)}\ge \|I_\mathcal Af\|_{\ell_1^N}.
\end{equation}

Consider a product space $X=X_1\times\cdots\times X_d$ with a product measure
$\mu=\mu_1\times\cdots\times\mu_d$.
Define the error of best $m$-term approximation of a function $f\colon X\to\R$
in $L_p(X,\mu)$ by tensor products:
\begin{equation}
    \label{eq_theta}
\Theta_m(f,L_p) := \inf_{u^{i,s}\colon X_i\to\mathbb R} \|f - \sum_{s=1}^m u^{1,s}(x_1)\cdots
u^{d,s}(x_d)\|_{L_p(X,\mu)}
\end{equation}
and $\Theta_m(\mathcal F,L_p) := \sup_{f\in\mathcal F}\Theta_m(f,L_p)$.

We say that $\mathcal A$ is a $(n_1\times\ldots\times n_d)$-product family of sets if
$$
\mathcal A = \{A_{i_1}^{(1)}\times\cdots\times A_{i_d}^{(d)}\colon 
i_1\in[n_1],\ldots,i_d\in[n_d]\}
$$
for some $A^{(s)}_{i_s}\subset X_s$. It is obvious that $|\mathcal A|=n_1\cdots
n_d$.

\begin{statement}
    \label{stm_f_low}
    Suppose that a class $\mathcal F\subset L_1(X,\mu)$
    $\gamma$-shatters a $(n_1\times\ldots\times n_d)$-product family $\mathcal
    A$ of disjoint sets. Then for $m\le c(d)N/\max\{n_j\}$,
    $N:=|\mathcal A|$, we have
    $$
    \Theta_m(\mathcal F,L_1(X,\mu)) \ge c(d)\gamma N.
    $$
\end{statement}

\begin{proof}
    Consider the discretization operator $\mathcal{I}_{\mathcal A}\colon
    L_1(X,\mu)\to\R^{n_1\times\ldots\times n_d}$. It
    maps rank-one functions $u^1(x_1)\ldots u^d(x_d)$ to rank-one tensors.
    By the shattering property, for any signum-tensor $\sigma\in\R^{n_1\times\cdots\times n_d}$ there is a
    function $f_\sigma\in\mathcal F$ such that the tensor $T(\sigma)=\mathcal{I}_{\mathcal
    A}f_\sigma$ has $\sign T(\sigma)=\sigma$ and $\min_I|T(\sigma)_I|\ge\gamma$.
    We apply Statement~\ref{stm_low} with some small $\eps$ to find
    a signum tensor $\sigma^*\in\mathcal S^*$ and then use~(\ref{I_bound})
    and~\eqref{eq_low}:
    $$
    \|f_{\sigma^*}-\sum_{s=1}^m u^{1,s}(x_1)\cdots u^{d,s}(x_d)\|_{L_1} \ge
    \|T(\sigma^*)-\widetilde{T}_m\|_{\ell_1^N} \ge \eps N\gamma.
    $$
\end{proof}

We will give some corollaries for $d$-variate functions.
Consider the space $C^r[0,1]^d$ of $r$-smooth functions
$f\colon[0,1]^d\to\mathbb R$ and the unit ball in this space:
$$
U(C^r[0,1]^d) = \{f\in C^r[0,1]^d\colon \max_{\|\alpha\|_1\le r} \|D^\alpha
f\|_\infty \le 1\}.
$$

\begin{statement}
    $$
    \Theta_m(U(C^r[0,1]^d),L_q) \asymp m^{-\frac{r}{d-1}},\quad 1\le q\le\infty.
    $$
\end{statement}
\begin{proof}
    We should establish the lower bound for $U(C^r)$ in $L_1$ metric.
    Pick some $n\in\mathbb N$ and divide the cube $[0,1]^d$ into $n^d$
    cubes:
    $$
    Q_I = \prod_{k=1}^d \left[\frac{i_k-1}{n},\frac{i_k}n\right],\quad
    I=(i_1,\ldots,i_d)\in [n]^d.
    $$
    A standard construction shows that our class
    $c(r,d)n^{-(r+d)}$-shatters $\mathcal A=\{Q_I\}_{I\in [n]^d}$.
    Indeed, take a ``hat'', i.e. infinitely smooth function
    $\varphi(x)$ with support in $(-\frac13,\frac13)$ and $\int \varphi(x)\,dx=1$.
Then $\max_{k\le r}\|\varphi^{(k)}\|_\infty \le c(r)$. 
Let $\varphi^{\otimes d}(x_1,\ldots,x_d)=\varphi(x_1)\cdots\varphi(x_d)$.
Given $\sigma\colon [n]^d\to\{-1,1\}$, we put
$$
    f_\sigma(x)=\sum_{I\in [n]^d} \sigma(I)\varphi^{\otimes d}(n(x-x_I)),
$$
where $x_I$ is the middle point of $Q_I$. As the supports of
    summands of $f_\sigma$ are disjoint,
    $\|f_\sigma\|_{C^r}\le c(r,d)n^r$, $\int_{Q_I}f_\sigma(x)\,dx = \sigma(I)n^{-d}$.

    Application of Statement~\ref{stm_f_low} gives us lower bound $\Theta_m\gg
    n^{-r}$ for $m\ll n^{d-1}$, as required.

The upper bound goes in $L_\infty$.
It is well-known that for $f\in C^r$ there is a function $g$ that is equal
    to a (Taylor) polynomial of degree $<r$ on each cube $Q_I$ and
    $\|f-g\|_\infty \ll n^{-r}$. (See, e.g.~\cite[\S8.4]{Zorich}.) Write $g$ as a sum $\sum_{\|\alpha\|_1<
    r}c_\alpha x^\alpha$, where $c_\alpha=c_\alpha(I)$ for $x\in Q_I$. For each
    $\alpha$ the coefficient function $c_\alpha$ may be viewed as a tensor in $\R^{n\times
    n\times\cdots\times n}$. It has rank $\le n^{d-1}$
    and hence $g$ has rank at most $c(r,d)n^{d-1}$. So, $\Theta_m\ll n^{-r}$ for
    $m\gg n^{d-1}$, as required.
\end{proof}

In order to obtain corollaries for classes of fractional smoothness, it is
convenient to work with trigonometric polynomials.
This research was motivated by the recent paper~\cite{BT15}, where lower bounds
for $\Theta_m^b$ were obtained. Here $\Theta_m^b\ge\Theta_m$ is the
error of the $m$-term approximation with the additional assumption that norms of
each term do not exceed $b$.

Let $\mathbf{n}=(n_1,\ldots,n_d)\in\mathbb N^d$ and denote by $\mathcal T(\mathbf{n})_\infty$
the set of real trigonometric polynomials of the form
$$
t(x) = \sum_{k\in\mathbb Z^d\colon |k_j|\le n_j}
c_ke^{i(k_1x_1+\ldots+k_dx_d)},\quad c_{-k}=\overline{c_k}
$$
with the property $\|t\|_\infty\le 1$.

It was proven by Temlyakov and Bazarkhanov,~\cite[Theorem 2.1]{BT15} that 
\begin{equation}
    \label{teml}
\Theta_m^b(\mathcal T(n,n\ldots,n),L_1)\ge C(b,d)>0\quad\mbox{if $m\log m\le
c(b,d)n^{d-1}$.}
\end{equation}
We slightly improve their result.

\def\thetheorem{\ref{thm_poly}}
\begin{theorem}
    For any $d\in\mathbb N$, $\mathbf{n}=(n_1,\ldots,n_d)$, $N:=n_1\cdots n_d$,
    we have
    \begin{equation}
        \label{our_poly}
    \Theta_m(\mathcal T(\mathbf{n})_\infty, L_1[-\pi,\pi]^d) \ge
        c_2(d)>0,\quad\mbox{if $m\le c_1(d)\frac{N}{\max\{n_j\}}$.}
    \end{equation}
\end{theorem}
Note that this bound is
sharp, because any $t\in\mathcal T(\mathbf{n})$ may be written as a sum of
$\asymp N/\max\{n_j\}$ rank-one functions:
$$
t(x) = \sum_{(k_2,\ldots,k_d)}\widetilde{t}_{k_2,\ldots,k_d}(x_1)e^{ik_2x_2}\cdots
e^{ik_dx_d}.
$$

We will make use of several inequalities for $\mathcal T(\mathbf{n})$,
see~\cite{DTU18}, \S2.4. The Bernstein inequality states that
\begin{equation}
    \label{bernstein}
    \|\frac{\partial}{\partial x_j}t\|_p\le c\|t\|_p n_j,\quad j\in[d],\; 1\le p\le\infty.
\end{equation}
Consider a regular grid ($\mathbf{m}\in\mathbb N^d$)
$$
\mathcal X_{\mathbf{m}} := \{(\frac{\pi k_1}{m_1},\ldots,\frac{\pi k_d}{m_d})\colon
-m_i<k_i\le m_i\}
$$
and the space $\ell_p(\mathcal X_{\mathbf{m}})$ of values of polynomials on
that grid.
The Marcinkiewicz discretization theorem states that for
$t\in\mathcal T(\mathbf{n})$ and $1\le p\le\infty$,
\begin{equation}
    \label{marcinkiewicz}
    c(d)N^{-1/p}\|t\|_{\ell_p(\mathcal X_{2\mathbf{n}})} \le \|t\|_p \le
    C(d)N^{-1/p}\|t\|_{\ell_p(\mathcal X_{2\mathbf{n}})},
\end{equation}

\begin{proof}[Proof of Theorem~\ref{thm_poly}]
    We may assume that all $n_j$ are rather large. Indeed, for any
    subsequence $n_{i_1},\ldots,n_{i_k}$ of $\mathbf{n}$ we have
    $$
    \Theta_m(\mathcal T_\infty(\mathbf{n}),L_1[-\pi,\pi]^d) \ge
    \Theta_m(\mathcal T_\infty(n_{i_1},\ldots,n_{i_k}),L_1[-\pi,\pi]^k)
    $$
    so we can get rid of small $n_j$.

    As in the previous proof, we will show that our class shatters some product
    family. Let $\alpha,\delta>0$ be some small reals. Consider the grid
    $$
    \mathcal X^\alpha:=\mathcal X_{\mathbf{n}'},\quad (n_1',\ldots,n_d'):=(\lfloor
    \alpha n_1\rfloor,\ldots,\lfloor \alpha n_d\rfloor)
    $$
    and the corresponding family of parallelopipeds $Q_x^\delta :=
    x+\delta\prod_{j=1}^d[-n_j^{-1},n_j^{-1}]$,
    $x\in\mathcal X^\alpha$. They are disjoint for small $\delta$. We are going to
    prove that $\mathcal T(\mathbf{n})$ shatters $\{Q_x^\delta\}_{x\in\mathcal
    X^\alpha}$ if $\alpha$ and
    $\delta$ are sufficiently small.

    Instead of the ``hat'' function $\varphi$ we utilize the Fejer's kernel
    $$
    K_n(x) := \frac1n\frac{\sin^2(nx/2)}{\sin^2(x/2)}.
    $$
    We will use that
    $$
    \frac1n K_n(0)=1,\quad \frac1n|K_n(x)|\le\min(1,\frac{\pi^2}{n^2x^2}),\;x\in[-\pi,\pi].
    $$
    Let 
    $$
    \psi(x) := \frac{K_{n_1}(x_1)}{n_1}\frac{K_{n_2}(x_2)}{n_2}\cdots \frac{K_{n_d}(x_d)}{n_d}.
    $$
    Given signs $\sigma\colon\mathcal X^\alpha\to\{-1,1\}$, we define the polynomial
    $$
    t_\sigma(y) := \sum_{x\in\mathcal X^\alpha} \sigma(x)\psi(y-x).
    $$

    Let us estimate $t_\sigma(x^*)$ for $x^*\in \mathcal X^\alpha$. The main term for
    $t_\sigma(x^*)$ is $\sigma(x^*)\psi(0) = \sigma(x^*)$. The contribution of
    other terms is
    \begin{multline*}
    |\sum_{x\in\mathcal X^\alpha\setminus\{x^*\}} \sigma(x)\psi(x^*-x)| \le
    \sum_{x\in\mathcal X^\alpha\setminus\{0\}} |\psi(x)| \le \\
        \le \sum_{\substack{|k_i|\le n_i'\\k\ne0}} \frac{1}{n_1}K_{n_1}(\frac{\pi
        k_1}{n_1'})\cdots \frac{1}{n_d}K_d(\frac{\pi k_d}{n_d'}) \le \sum_{k\in\mathbb
        Z^d\setminus\{0\}}\prod_{j=1}^d \min(1,\frac{c\alpha^2}{k_j^2}).
    \end{multline*}
    It is easy to see that the last expression is bounded by $1/2$ for small
        $\alpha=\alpha(d)$. Fix such $\alpha$. Therefore, $\sigma(x^*)t_\sigma(x^*)\ge1/2$.

        Analogously, one can show that $t_\sigma$ is bounded by some $c(d)$ on
        the grid $\mathcal X_{2\mathbf{n}}$; hence, by the discretization theorem~(\ref{marcinkiewicz}),
        we obtain $\|t\|_\infty\le c_1(d)$. 

        Finally, using~(\ref{bernstein}) for $p=\infty$, we see that
        $\sigma(x^*)t_\sigma(x)\ge1/3$ for $x\in Q_{x^*}^\delta$ if $\delta$ is
        small. It follows that the family $\{Q_x^\delta\}_{x\in\mathcal X^\alpha}$
        is $\gamma$-shattered by $\mathcal T(\mathbf{n})_\infty$ with
        $\gamma\asymp\int_{Q_{x^*}^\delta} 1 \asymp N^{-1}$.
    It remains to use Statement~\ref{stm_f_low}.
\end{proof}

The lower bound for $\mathcal T(\mathbf{n})$ allows us to get the correct orders
of decay for classes of dominated mixed smoothness $W^{\mathbf{r}}_p$,
$\mathbf{r}=(r,r,\ldots,r)$.
See~\cite{BT15},
where lower bounds for $\Theta_m^b(W^{\mathbf{r}}_p,L_q)$
were obtained  using~\eqref{teml}. We apply~\eqref{our_poly} instead.
Using the inclusion $n^{-rd}\mathcal T(n,n,\ldots,n)_\infty \subset
cW^{\mathbf{r}}_\infty$ and~\eqref{our_poly}, we obtain
$$
\Theta_m(W^{\mathbf{r}}_\infty,L_1[0,1]^d) \gg m^{-\frac{rd}{d-1}}
$$
The corresponding upper bound is given in~\cite[Theorem 4.1]{T89}:
$$
\Theta_m(W^{\mathbf{r}}_2,L_2[0,1]^d) \ll m^{-\frac{rd}{d-1}}.
$$
Putting this together, we obtain Theorem~\ref{th_mixed}.


\begin{thebibliography}{XXXX}
    \bibitem[V77]{V77}
        L.G.~Valiant,
        ``Graph-theoretic arguments in low-level complexity'',
        \textit{Mathematical Foundations of Computer Science (MFCS)}, 1977,
        p.162--176. Berlin, Heidelberg,

    \bibitem[AW17]{AW17}
        J.~Alman, R.~Williams,
        ``Probabilistic rank and matrix rigidity'',
         \textit{STOC 2017: Proceedings of the 49th Annual ACM SIGACT Symposium on
         Theory of Computing}, June 2017, p.641--652.

    \bibitem[BT15]{BT15}
        D. Bazarkhanov, V. Temlyakov,
        ``Nonlinear tensor product approximation of functions'',
        \textit{J. Complexity}, \textbf{31}:6 (2015), 867--884.

    \bibitem[AFR85]{AFR85}
        N. Alon, P. Frankl, V. R\"odl,
        ``Geometrical realization of set systems and probabilistic communication
        complexity'',
        \textit{Proc. 26th Ann. Symposium on Foundations of Computer Science}
        (1985).

    \bibitem[ALSV13]{ALSV13}
        N.~Alon, T.~Lee, A.~Shraibman, S.~Vempala,
        ``The approximate rank of a matrix and it's algorithmic applications'',
        \textit{Proc. of the 2013 ACM Symposium on Theory of Computing}, 675--684.

    \bibitem[BW01]{BW01}
        H. Buhrman, R. de Wolf,
        ``Communication complexity lower bounds by polynomials'',
        \textit{Proceedings 16th Annual IEEE Conference on Computational
        Complexity}, 2001, pp. 120-130.

    \bibitem[DTU18]{DTU18}
        D.~Dung, V.~Temlyakov, T.~Ullrich,
        \textit{Hyperbolic Cross Approximation}.
        Birkh\"auser, Cham. 2018.

    \bibitem[LGM96]{LGM96}
        G.G.~Lorentz, M.~Golitschek, Y.~Makovoz,
        \textit{Constructive Approximation: Advanced Problems},
        1996.

    \bibitem[P85]{P85}
        A.~Pinkus,
        \textit{n-Widths in Approximation Theory},
        1985.

    \bibitem[GES]{GES}
        B.~Golubov, A.~Efimov, V.~Skvortsov,
        \textit{Walsh Series and Transforms}, 1991.

    \bibitem[L08]{L08}
        S.V.~Lokam,
        ``Complexity Lower Bounds using Linear Algebra'',
        \textit{Foundations and Trends in Theoretical Computer Science},
        \textbf{4}:1-2 (2008), p.1-155.

    \bibitem[T89]{T89}
        V.~Temlyakov,
        ``Estimates of best bilinear approximations of periodic functions'',
        \textit{Proc. Steklov Inst.}, \textbf{181} (1989), 275--293.

    \bibitem[T92]{T92}
        V.~Temlyakov,
        ``Estimates of Best Bilinear Approximations of Functions and
        Approximation Numbers of Integral Operators'',
        \textit{Mathematical Notes}, \textbf{51} (1992), 510--517.

    \bibitem[T03]{T03}
        V.N.~Temlyakov,
        ``Nonlinear Methods of Approximation'',
        \textit{Found. Comput.  Math.}, \textbf{3} (2003), 33--107.

    \bibitem[ST89]{ST89}
        S.J.~Szarek, M.~Talagrand,
        ``An ``isomorphic'' version of the Sauer--Schelah lemma and the Banach--Mazur
        distance to the cube'',
        \textit{Geometric Aspects of Functional Analysis. Lecture Notes in
        Mathematics}, v.\textbf{1376} (1989).

    \bibitem[W68]{W68}
        H.E.~Warren,
        ``Lower bounds for approximation by nonlinear manifolds'',
        \textit{Trans. AMS}, \textbf{133}:1 (1968), 167--178.

    \bibitem[KMR18]{KMR18}
        B.S.~Kashin, Yu.V.~Malykhin, K.S.~Ryutin,
        ``Kolmogorov Width and Approximate Rank'',
        \textit{Proc. Steklov Inst.}, \textbf{303} (2018), 140--153.

    \bibitem[DL19]{DL19}
        Z.~Dvir, A.~Liu,
        ``Fourier and Circulant Matrices Are Not Rigid'',
        \textit{34th Computational Complexity Conference} (2019),
        pp.17:1--17:23.

    \bibitem[LS09]{LS09}
        T. Lee and A. Shraibman,
        ``An approximation algorithm for approximation rank'',
        \textit{Proc. 24th Annu. IEEE Conf. on Computational Complexity}, Paris,
        (2009), pp. 351--357.

    \bibitem[PS86]{PS86}
        R.~Paturi, J.~Simon,
        ``Probabilistic communication complexity'',
        \textit{Journal of Computer and System Sciences}, vol. \textbf{33}, no.
        1, pp. 106--123.

    \bibitem[VHA15]{VHA15}
        V.~Vovk, H.~Papadopoulos, A.~Gammerman (editors),
        \textit{Measures of Complexity. Festschrift for Alexey Chervonenkis},
        Springer, 2015.

    \bibitem[Tikh87]{Tikh87}
        V.M.~Tikhomirov, ``Approximation theory'' (in Russian),
        \textit{Itogi Nauki i Tekhniki. Ser. Sovrem. Probl. Mat. Fund. Napr.,},
        \textbf{14}, VINITI, Moscow, 1987, 103–260

    \bibitem[Zorich]{Zorich}
        V.A.~Zorich,
        \textit{Mathematical Analysis I},
        Springer, 2004.

\end{thebibliography}
\end{document}